\documentclass[12pt,reqno]{amsart}

\usepackage{amssymb}

\usepackage[all]{xy}

\numberwithin{equation}{section}

\newtheorem{theorem}{Theorem}[section]
\newtheorem{proposition}[theorem]{Proposition}
\newtheorem{lemma}[theorem]{Lemma}
\newtheorem{corollary}[theorem]{Corollary}

\theoremstyle{definition}

\newtheorem{assumption}[theorem]{Assumption}

\newcommand{\CP}{\mathbb{CP}} \newcommand{\CC}{\mathbb{C}}

 \newcommand{\scrH}{\mathcal{H}}
\newcommand{\cO}{\mathcal{O}}

\DeclareMathOperator*{\PGL}{PGL}
\DeclareMathOperator*{\SL}{SL}

\begin{document}

\baselineskip=15pt

\title[Equivariant vector bundles on $\mathbb{CP}^1$]{On the equivariant vector bundles on $\mathbb{CP}^1$}

\author[I. Biswas]{Indranil Biswas}

\address{Department of Mathematics, Shiv Nadar University, NH91, Tehsil
Dadri, Greater Noida, Uttar Pradesh 201314, India}

\email{indranil.biswas@snu.edu.in, indranil29@gmail.com}

\author[M. Kumar]{Manish Kumar}

\address{Statistics and Mathematics Unit, Indian Statistical Institute,
Bangalore 560059, India}

\email{manish@isibang.ac.in}

\author[A. J. Parameswaran]{A. J. Parameswaran}

\address{Kerala School of Mathematics, Kunnamangalam PO, Kozhikode, Kerala, 673571, India}

\email{param@ksom.res.in}

\subjclass[2010]{14H60, 14F06, 14L30}

\keywords{Projective line, equivariant bundle, semistability, Harder-Narasimhan filtration}

\date{}

\begin{abstract}
Let $H$ be a subgroup of $\PGL(2,\CC)$ (respectively, ${\rm SL}(2,\CC)$) such that the Zariski closure
in $\PGL(2,\CC)$ (respectively, ${\rm SL}(2,\CC)$) of some compact subgroup of $H$ contains $H$.
We classify the $H$--equivariant holomorphic vector bundles on $\mathbb{CP}^1$. This generalizes
\cite{BM} where $H$ was assumed to be a finite abelian group.
\end{abstract}

\maketitle

\section{Introduction}

Given a finite abelian group $H$ with a faithful action on $\mathbb{CP}^1$, in \cite{BM} a complete description of all
$H$-equivariant vector bundles on $\mathbb{CP}^1$ was given. As mentioned in \cite{BM}, it was very much inspired by
\cite{DF}. In this paper, we extend the result to any closed topological subgroup $H$ of $\PGL(2,\CC)$ (respectively, ${\rm SL}(2,\CC)$) satisfying the hypothesis that it contains a compact subgroup $K_H$ whose Zariski closure in $\PGL(2,\CC)$ (respectively, ${\rm SL}(2,\CC)$) contains $H$. All finite subgroups
of $\PGL(2,\CC)$ (respectively, ${\rm SL}(2,\CC)$) clearly satisfy this property.

Note that $\PGL(2,\CC)\cong \SL(2,\CC)/{\pm I}$. Let $G$ be a closed subgroup of $\SL(2,\CC)$ such that the 
Zariski closure of a compact subgroup $K_G$ contains $G$. Consider the action of $G$ on $\CP^1$ coming from the 
$SL(2,\CC)$-action.

Let $G$ be a topological subgroup of ${\rm SL}(2,\CC)$ such that there is a compact subgroup $K_G\, \subset\, G$ whose
Zariski closure in ${\rm SL}(2,\CC)$ contains $G$. we prove the following:

\begin{enumerate}
 \item Every $G$-equivariant vector bundle on $\CP^1$ splits holomorphically as a direct sum of $G$--equivariant line bundles.

 \item Every vector bundle on $\CP^1$ has a $G$-equivariant structure.

 \item Any $G$--equivariant vector bundle on $\CP^1$ is isomorphic to one of the form
$\bigoplus_{i=1}^r{\cO}_{\CP^1}(d_i)\otimes_{\CC}M_i$, where each $M_i$ is finite dimensional $\CC[G]$--module and $\cO(d_i)$
is equipped with the natural $G$--equivariant structure.
\end{enumerate}

Let $H$ be a closed subgroup of $\PGL(2,\CC)$ satisfying the above hypothesis. There is a dichotomy.
Let $\scrH$ be the inverse image of $H$ in $\SL(2,\CC)$ for the quotient map $\SL(2,\CC)\,
\longrightarrow\, \PGL(2,\CC)$. Consider the short exact sequence of groups
$$1\,\longrightarrow\, \{\pm I\}\,\longrightarrow\,\scrH\,\longrightarrow\, H\,\longrightarrow\, 1.$$

If the above short exact sequence splits, then we are in the above setup, i.e., $H$ is a subgroup of $\SL(2,\CC)$, and
the above results hold. If the short exact sequence does not split, then any $H$--equivariant vector bundle on $\CP^1$
is isomorphic to one of the form $$
\left(\bigoplus_{i=1}^r {\mathcal O}_{{\mathbb C}{\mathbb P}^1}(2d_i)\otimes_{\mathbb C} M_i\right)
\bigoplus \left(\bigoplus_{i=1}^s {\mathcal O}_{{\mathbb C}{\mathbb P}^1}(2\delta_i+1)\otimes_{\mathbb C} N_i\right),
$$
where each $M_i$ (respectively, $N_i$) is a finite dimensional $\CC[H]$--module (respectively, $\CC[\scrH]$--module such that the action of the nontrivial central element $-I\,\in \, \mathcal H$ on $N_i$ is multiplication by $-1$.

\section{Automorphisms of ${\mathbb C}{\mathbb P}^1$ and equivariant bundles} 

\subsection{Equivariant vector bundles}

Consider the complex projective line ${\mathbb C}{\mathbb P}^1$ parametrizing the one-dimensional quotients
of ${\mathbb C}^2$. The standard action of $\text{SL}(2,{\mathbb C})$ on ${\mathbb C}^2$ produces an action
of $\text{SL}(2,{\mathbb C})$
\begin{equation}\label{e1}
\rho\ :\ \text{SL}(2,{\mathbb C})\times{\mathbb C}{\mathbb P}^1\ \longrightarrow\ {\mathbb C}{\mathbb P}^1
\end{equation}
on ${\mathbb C}{\mathbb P}^1$. Let
\begin{equation}\label{e2}
\widehat{\rho}\ :\ \text{SL}(2,{\mathbb C})\ \longrightarrow\ \text{Aut}({\mathbb C}{\mathbb P}^1)
\end{equation}
be the homomorphism given by $\rho$, where $\text{Aut}({\mathbb C}{\mathbb P}^1)$ is the group of 
all holomorphic automorphisms of ${\mathbb C}{\mathbb P}^1$. Since $\widehat{\rho}(\pm I)\,=\,
{\rm Id}_{{\mathbb C}{\mathbb P}^1}$, the homomorphism $\widehat{\rho}$ produces a homomorphism
\begin{equation}\label{e3}
\widehat{\sigma}\, :\, \text{PGL}(2,{\mathbb C})\, :=\, \text{SL}(2,{\mathbb C})/\{\pm I\}\, \longrightarrow\,
\text{Aut}({\mathbb C}{\mathbb P}^1)
\end{equation}
which is actually an isomorphism. Let
\begin{equation}\label{e4}
\sigma\ :\ \text{PGL}(2,{\mathbb C})\times{\mathbb C}{\mathbb P}^1\ \longrightarrow\ {\mathbb C}{\mathbb P}^1
\end{equation}
be the action of $\text{PGL}(2,{\mathbb C})$ on ${\mathbb C}{\mathbb P}^1$ given by the homomorphism $\widehat{\sigma}$ 
in \eqref{e3}.

Let $$G\, \subset\, \text{SL}(2,{\mathbb C})$$ be a closed topological subgroup. It should be clarified
that $G$ need not be a complex subgroup. Also, $G$ need not be connected. 

\begin{assumption}\label{assm}
We assume that there is a compact subgroup $K_G\, \subset\, G$ such that Zariski closure of $K_G$ in
${\rm SL}(2,{\mathbb C})$ contains $G$.
\end{assumption}

Note that any finite subgroup of $\text{SL}(2,{\mathbb C})$ satisfies Assumption \ref{assm}. More generally, any
reductive subgroup of $\text{SL}(2,{\mathbb C})$ satisfies Assumption \ref{assm}.

A $G$--\textit{equivariant} vector bundle is a holomorphic vector bundle $p\, :\, E\, \longrightarrow\, {\mathbb C}{\mathbb P}^1$
equipped with an action
\begin{equation}\label{e5}
\varphi\ :\ G\times E\ \longrightarrow\ E
\end{equation}
such that
\begin{enumerate}
\item $p\circ \varphi (g,\, v)\,=\, \rho(g,\, p(v))$ for all $(g,\, v)\, \in\, G\times E$, where $\rho$ is the map in \eqref{e1}, and

\item the map $E\, \longrightarrow\, E$ defined by $v\,\longmapsto\, \varphi (g,\, v)$ is a holomorphic isomorphism of vector bundles
over the automorphism $\widehat{\rho}(g)$, for all $g\, \in\, G$, where $\widehat{\rho}$ is the homomorphism in \eqref{e2}.
\end{enumerate}

For any $g\, \in\, G$ and $v\, \in\, E$, the notation
\begin{equation}\label{n1}
g\cdot v \ := \ \varphi(g,\, v)
\end{equation}
will be used for convenience.

Let ${\mathcal O}_{{\mathbb C}{\mathbb P}^1}(1)$ be the tautological line bundle on ${\mathbb C}{\mathbb P}^1$. The
standard action of $\text{SL}(2,{\mathbb C})$ on ${\mathbb C}^2$ produces an action of
$\text{SL}(2,{\mathbb C})$ on ${\mathcal O}_{{\mathbb C}{\mathbb P}^1}(1)$. So ${\mathcal O}_{{\mathbb C}{\mathbb P}^1}(1)$
is a $G$--equivariant line bundle. Hence ${\mathcal O}_{{\mathbb C}{\mathbb P}^1}(n)\,=\,
{\mathcal O}_{{\mathbb C}{\mathbb P}^1}(1)^{\otimes n}$ is also $G$--equivariant for all $n\, \in\, {\mathbb Z}$.

\subsection{Equivariant for a subgroup of $\text{SL}(2,{\mathbb C})$}

Take a $G$--equivariant vector bundle $p\, :\, E\, \longrightarrow\, {\mathbb C}{\mathbb P}^1$. Let
\begin{equation}\label{e6}
0 \,=\, E_0\,\subset \, E_1\,\subset\, \cdots \, \subset\, E_{\ell-1}\, \subset\, E_\ell\,=\, E
\end{equation}
be the Harder--Narasimhan filtration of $E$ (see \cite[p.~16, Theorem 1.3.4]{HL} for Harder--Narasimhan filtration).

\begin{lemma}\label{lem1}
The action of $G$ on $E$ (see \eqref{e5}) preserves the filtration in \eqref{e6}.
\end{lemma}

\begin{proof}
Let $\varphi$ be the action of $G$ on $E$ (see \eqref{e5}). For any $g\, \in \, G$, the filtration
$$
0 \,=\, \varphi(g,\, E_0)\,\subset \, \varphi(g,\, E_1)\,\subset\, \cdots \, \subset\, \varphi(g,\, E_{\ell-1})
\, \subset\, \varphi(g,\, E_\ell)\,=\, E
$$
clearly satisfies all the conditions for a Harder--Narasimhan filtration. So the lemma follows immediately from the
uniqueness of the Harder--Narasimhan filtration (see \cite[p.~16, Theorem 1.3.4]{HL}).
\end{proof}

\begin{proposition}\label{prop1}
Let
$$
0\, \longrightarrow\, A \, \longrightarrow\, B\, \stackrel{q}{\longrightarrow}\, C \, \longrightarrow\, 0
$$
be a short exact sequence of $G$--equivariant vector bundles on ${\mathbb C}{\mathbb P}^1$. If this
short exact sequence splits holomorphically, then it admits a holomorphic $G$--equivariant splitting. 
\end{proposition}

\begin{proof}
Assume that the short exact sequence in the proposition splits holomorphically. Fix a holomorphic splitting
$$
\psi\ :\ C\ \longrightarrow\ B,
$$
so we have $q\circ\psi \,=\, {\rm Id}_C$, where $q$ is the projection in the proposition. For any $g\, \in\, G$,
it is evident that the homomorphism
$$
\psi_g \ :\ C\ \longrightarrow\ B, \ \ \ v \, \longmapsto\, g^{-1}\cdot \psi (g\cdot v)
$$
(see \eqref{n1}) also satisfies the condition $q\circ\psi_g \,=\, {\rm Id}_C$.

Consider the compact subgroup $K_G\, \subset\, G$ in Assumption \ref{assm}. Construct the homomorphism
$$
\widetilde{\psi} \ :\ C\ \longrightarrow\ B, \ \ \ v \ \longmapsto\ \int_{K_G} \psi_g(v)\mu_{K_G},
$$
where $\mu_{K_G}$ is the Haar measure of $K_G$. It is straightforward to check that $\widetilde{\psi}$ is
$K_G$--equivariant and it is holomorphic. Since the Zariski closure of $K_G\, \subset\, \text{SL}(2,{\mathbb C})$
contains $G$, it follows that $\widetilde{\psi}$ is actually a holomorphic $G$--equivariant
splitting of the exact sequence in the proposition.
\end{proof}

In Lemma \ref{lem1} we saw that \eqref{e6} is a filtration of $G$--equivariant vector bundles.

\begin{corollary}\label{cor1}
The filtration in \eqref{e6} admits a holomorphic $G$--equivariant splitting.
\end{corollary}

\begin{proof}
All holomorphic vector bundles $W$ on ${\mathbb C}{\mathbb P}^1$ of rank $r$ are of the form
$$
W \ = \ \bigoplus_{i=1}^r {\mathcal O}_{{\mathbb C}{\mathbb P}^1}(d_i)
$$
\cite[p.~122, Th\'eor\`eme 1.1]{Gr}. Also, we have
$$
H^1({\mathbb C}{\mathbb P}^1, \, \text{Hom}({\mathcal O}_{{\mathbb C}{\mathbb P}^1}(a),\,
{\mathcal O}_{{\mathbb C}{\mathbb P}^1}(a+b)))\ =\ 0
$$
for all $b\, \geq\, 0$ and $a\, \in\, {\mathbb Z}$. Consequently, the filtration in \eqref{e6} admits a holomorphic splitting.
Now Proposition \ref{prop1} says that it admits a holomorphic $G$--equivariant splitting.
\end{proof}

Let
\begin{equation}\label{e7}
W \ =\ {\mathcal O}_{{\mathbb C}{\mathbb P}^1}(d)^{\oplus r}
\end{equation}
be a $G$--equivariant vector bundle on ${\mathbb C}{\mathbb P}^1$. As noted before, using the standard
action of $\text{SL}(2,{\mathbb C})$ on ${\mathbb C}^2$ we get a natural $G$--equivariant structure on 
the line bundle ${\mathcal O}_{{\mathbb C}{\mathbb P}^1}(d)$. Let $M$ be a finite dimensional complex
$G$--module. The actions of $G$ on ${\mathcal O}_{{\mathbb C}{\mathbb P}^1}(d)$ and $M$ combine together
to produce an action of $G$ on the holomorphic vector bundle ${\mathcal O}_{{\mathbb C}{\mathbb P}^1}(d)
\bigotimes_{\mathbb C} M$. In other words,
\begin{equation}\label{e8}
{\mathcal O}_{{\mathbb C}{\mathbb P}^1}(d)\otimes_{\mathbb C} M \ \longrightarrow\ {\mathbb C}{\mathbb P}^1
\end{equation}
is a $G$--equivariant vector bundle.

\begin{proposition}\label{prop2}
There is a finite dimensional complex $G$--module $M$ such that the $G$--equivariant vector bundle $W$ in
\eqref{e7} is holomorphically and $G$--equivariantly isomorphic to ${\mathcal O}_{{\mathbb C}{\mathbb P}^1}(d)
\bigotimes_{\mathbb C} M$ in \eqref{e8}.
\end{proposition}

\begin{proof}
Consider the complex vector space
$$
M\ :=\ H^0({\mathbb C}{\mathbb P}^1, \, \text{Hom}({\mathcal O}_{{\mathbb C}{\mathbb P}^1}(d),\, W)).
$$
The actions of $G$ on ${\mathcal O}_{{\mathbb C}{\mathbb P}^1}(d)$ and $W$ together produce an action of
$G$ on $M$. Note that there is the natural evaluation homomorphism
$$
\Psi\ :\ {\mathcal O}_{{\mathbb C}{\mathbb P}^1}(d)\otimes_{\mathbb C} M\ \longrightarrow\, W
$$
that sends any $w\otimes m$, where $w\, \in\, {\mathcal O}_{{\mathbb C}{\mathbb P}^1}(d)_x$,
$x\, \in\, {\mathbb C}{\mathbb P}^1$, and $m\, \in\, M$, to $m(w)\, \in\, W_x$.

{}From the construction of the action of $G$ on $M$ it follows immediately that $\Psi$ is
$G$--equivariant. Also, the homomorphism $\Psi$ is evidently holomorphic. Since $W$ is a direct sum of
copies of ${\mathcal O}_{{\mathbb C}{\mathbb P}^1}(d)$, we conclude that $\Psi$ is an isomorphism.
\end{proof}

\begin{theorem}\label{thm1}
Any $G$--equivariant vector bundle $W$ on ${\mathbb C}{\mathbb P}^1$ is holomorphically and $G$--equivariantly isomorphic
to a $G$--equivariant vector bundle of the form
$$
\bigoplus_{i=1}^r {\mathcal O}_{{\mathbb C}{\mathbb P}^1}(d_i)\otimes_{\mathbb C} M_i,
$$
where each $M_i$ is a finite dimensional complex $G$--module.

Moreover, the collection $\{d_i,\, M_i\}_{i=1}^r$ is uniquely determined by the isomorphism class of the
$G$--equivariant vector bundle $W$.
\end{theorem}

\begin{proof}
{}From Corollary \ref{cor1} and Proposition \ref{prop2} we conclude that there is a positive integer $r$,
integers $\{d_i\}_{i=1}^r$ and finite dimensional complex $G$--modules $\{M_i\}_{i=1}^r$, such that $W$ is holomorphically
and $G$--equivariantly isomorphic to the $G$--equivariant vector bundle
$\bigoplus_{i=1}^r {\mathcal O}_{{\mathbb C}{\mathbb P}^1}(d_i)\bigotimes_{\mathbb C} M_i$.

To see the uniqueness, take any $G$--equivariant vector bundle
$$
{\mathcal W}\ :=\ \bigoplus_{i=1}^s {\mathcal O}_{{\mathbb C}{\mathbb P}^1}(\delta_i)\otimes_{\mathbb C} N_i,
$$
where each $N_i$ is a finite dimensional complex $G$--module. Note that $s$ is the length of the Harder--Narasimhan
filtration of $\mathcal W$. The integers $\delta_i$ are the slopes of the successive quotients of the
Harder--Narasimhan filtration of $\mathcal W$. The $G$--module $N_i$ is identified with the space of global homomorphisms
from ${\mathcal O}_{{\mathbb C}{\mathbb P}^1}(\delta_i)$ to the $i$--th successive quotient of the
Harder--Narasimhan filtration of $\mathcal W$. From these the uniqueness statement follows immediately.
\end{proof}

\section{Equivariant for a subgroup of $\text{PGL}(2,{\mathbb C})$}

Let $$H\, \subset\, \text{PGL}(2,{\mathbb C})$$ be a closed topological subgroup. As before,
$H$ need not be connected or be a complex subgroup.

\begin{assumption}\label{assm2}
We assume that there is a compact subgroup $K_H\, \subset\, H$ such that Zariski closure of $K_H$ in
${\rm PGL}(2,{\mathbb C})$ contains $H$.
\end{assumption}

An $H$--\textit{equivariant} vector bundle is a holomorphic vector bundle $p\, :\, E\, \longrightarrow\, {\mathbb C}{\mathbb P}^1$
equipped with action
\begin{equation}\label{e9}
\varphi\ :\ H\times E\ \longrightarrow\ E
\end{equation}
such that
\begin{enumerate}
\item $p\circ \varphi (g,\, v)\,=\, \sigma(g,\, p(v))$ for all $(g,\, v)\, \in\, H\times E$, where $\sigma$ is the map in
\eqref{e4}, and

\item the map $E\, \longrightarrow\, E$ defined by $v\,\longmapsto\, \varphi (g,\, v)$ is a holomorphic isomorphism of vector bundles
over the automorphism $\widehat{\sigma}(g)$, for all $g\, \in\, H$, where $\widehat{\sigma}$ is the homomorphism in \eqref{e3}.
\end{enumerate}

The standard action of $\text{SL}(2,{\mathbb C})$ on ${\mathbb C}^2$ produces an action of
$\text{PGL}(2,{\mathbb C})$ on ${\mathcal O}_{{\mathbb C}{\mathbb P}^1}(2)$. This action of
$\text{PGL}(2,{\mathbb C})$ on ${\mathcal O}_{{\mathbb C}{\mathbb P}^1}(2)$ produces an
action of $\text{PGL}(2,{\mathbb C})$ on ${\mathcal O}_{{\mathbb C}{\mathbb P}^1}(2j)$ for
all $j\, \in\, {\mathbb Z}$. Thus the line bundle ${\mathcal O}_{{\mathbb C}{\mathbb P}^1}(2j)$
is $H$--equivariant in a natural way.

But the action of $\text{PGL}(2,{\mathbb C})$ on ${\mathbb C}{\mathbb P}^1$
does not lift to ${\mathcal O}_{{\mathbb C}{\mathbb P}^1}(1)$. Let
\begin{equation}\label{eb}
\beta\ :\ \text{SL}(2,{\mathbb C})\ \longrightarrow\ \text{PGL}(2,{\mathbb C})\ =\ \text{SL}(2,{\mathbb C})/\{\pm I\}
\end{equation}
be the quotient homomorphisms. We have a short exact sequence of groups
\begin{equation}\label{e10}
0\, \longrightarrow\, {\mathbb Z}/2{\mathbb Z}\, \longrightarrow\, {\mathcal H}\ :=\ \beta^{-1}(H)\, \stackrel{\beta'}{\longrightarrow}
\, H \, \longrightarrow\, 0,
\end{equation}
where $\beta'$ is the restriction of $\beta$ (see \eqref{eb}) to $\beta^{-1}(H)$. Since the kernel of $\beta'$ is in
the center of $\mathcal H$, the exact sequence in \eqref{e10} is left-split if there is a homomorphism
$$
\gamma\ :\ H\ \longrightarrow\ {\mathcal H}
$$
such that $\beta'\circ\gamma\,=\, {\rm Id}_H$.

\begin{lemma}\label{lem3}
The line bundle ${\mathcal O}_{{\mathbb C}{\mathbb P}^1}(1)$ has an $H$--equivariant structure if and only if
the exact sequence in \eqref{e10} splits.
\end{lemma}

\begin{proof}
Let $\gamma\, :\, H\, \longrightarrow\, {\mathcal H}$ be a splitting of \eqref{e10}. Consider the action of
$\text{SL}(2,{\mathbb C})$ on ${\mathcal O}_{{\mathbb C}{\mathbb P}^1}(1)$. Restrict it to the subgroup
$\gamma(H)\, \subset\, \text{SL}(2,{\mathbb C})$. This makes ${\mathcal O}_{{\mathbb C}{\mathbb P}^1}(1)$
an $H$--equivariant line bundle.

To prove the converse, assume that ${\mathcal O}_{{\mathbb C}{\mathbb P}^1}(1)$ has the structure of an
$H$--equivariant line bundle. This $H$--equivariance structure produces a homomorphism
\begin{equation}\label{e11}
\alpha\ :\ H\ \longrightarrow\ \textbf{Aut}({\mathcal O}_{{\mathbb C}{\mathbb P}^1}(1)),
\end{equation}
where $\textbf{Aut}({\mathcal O}_{{\mathbb C}{\mathbb P}^1}(1))$ is the group of holomorphic automorphisms of
${\mathcal O}_{{\mathbb C}{\mathbb P}^1}(1)$ over the holomorphic automorphisms of ${\mathbb C}{\mathbb P}^1$.
Now, $\textbf{Aut}({\mathcal O}_{{\mathbb C}{\mathbb P}^1}(1))$ coincides with $\text{SL}(2,{\mathbb C})$.
Using this identification of $\textbf{Aut}({\mathcal O}_{{\mathbb C}{\mathbb P}^1}(1))$ with $\text{SL}(2,{\mathbb C})$,
the homomorphism $\alpha$ in \eqref{e11} gives a homomorphism
$$
\alpha'\ :\ H\ \longrightarrow\ \text{SL}(2,{\mathbb C}).
$$
It is evident that $\beta(\alpha'(H))\,=\, H$, where $\beta$ is the projection in \eqref{eb}, and, moreover,
$\alpha'$ is a splitting of the exact sequence in \eqref{e10}.
\end{proof}

\begin{corollary}\label{cor2}
Take any $d\, \in\, \mathbb Z$. The line bundle ${\mathcal O}_{{\mathbb C}{\mathbb P}^1}(2d+1)$ has an $H$--equivariant
structure if and only if the exact sequence in \eqref{e10} splits.
\end{corollary}

\begin{proof}
It was noted that ${\mathcal O}_{{\mathbb C}{\mathbb P}^1}(2)$ has an $H$--equivariant
structure. Therefore, ${\mathcal O}_{{\mathbb C}{\mathbb P}^1}(2d+1)$ has an $H$--equivariant
structure if and only if ${\mathcal O}_{{\mathbb C}{\mathbb P}^1}(1)$ has an $H$--equivariant
structure. Now the result follows from Lemma \ref{lem3}.
\end{proof}

Take an $H$--equivariant vector bundle $p\, :\, E\, \longrightarrow\, {\mathbb C}{\mathbb P}^1$. Let
\begin{equation}\label{e6b}
0 \,=\, E_0\,\subset \, E_1\,\subset\, \cdots \, \subset\, E_{\ell-1}\, \subset\, E_\ell\,=\, E
\end{equation}
be the Harder--Narasimhan filtration of $E$.

\begin{proposition}\label{prop3}\mbox{}
\begin{enumerate}
\item The action of $H$ on $E$ (see \eqref{e9}) preserves the filtration in \eqref{e6b}.

\item The filtration in \eqref{e6b} admits a holomorphic $H$--equivariant splitting.
\end{enumerate}
\end{proposition}

\begin{proof}
The proof is identical to the proofs of Lemma \ref{lem1} and Corollary \ref{cor1}.
\end{proof}

Next we will construct some $H$--equivariant vector bundles on ${\mathbb C}{\mathbb P}^1$.

First assume that the exact sequence in \eqref{e10} splits.
Let $$\gamma\ :\ H\ \longrightarrow\ {\mathcal H}\ \subset\ \text{SL}(2,{\mathbb C})$$ be a splitting of \eqref{e10};
so $\beta'\circ\gamma\,=\, {\rm Id}_H$, where $\beta'$ is the homomorphism in \eqref{e10}.
Using $\gamma$, the line bundle ${\mathcal O}_{{\mathbb C}{\mathbb P}^1}(d)$ has a natural $H$--equivariance structure
for every $d\, \in\, \mathbb Z$. Let $M$ be a finite dimensional complex $H$--module. The actions of $H$ on
${\mathcal O}_{{\mathbb C}{\mathbb P}^1}(d)$ and $M$ together produce an action of $H$ on the holomorphic vector bundle
${\mathcal O}_{{\mathbb C}{\mathbb P}^1}(d)\bigotimes_{\mathbb C} M$. Thus ${\mathcal O}_{{\mathbb C}{\mathbb P}^1}(d)
\bigotimes_{\mathbb C} M$ has a natural $H$--equivariance structure.

Next assume that the exact sequence in \eqref{e10} \textit{does not} split. Then, as before, the holomorphic vector bundle
${\mathcal O}_{{\mathbb C}{\mathbb P}^1}(2d) \bigotimes_{\mathbb C} M$ has a natural $H$--equivariance structure
for every finite dimensional complex $H$--module $M$ and every $d\, \in\, \mathbb Z$.

Take any $d\, \in\, {\mathbb Z}$. Consider the natural action of ${\mathcal H}\, \subset\, \text{SL}(2,{\mathbb C})$
(the group in \eqref{e10}) on the line bundle ${\mathcal O}_{{\mathbb C}{\mathbb P}^1}(2d+1)$. It can be shown that the
nontrivial element of ${\mathbb Z}/2{\mathbb Z}\, \subset\, \mathcal H$ (see \eqref{e10}) acts as multiplication by
$-1$ on ${\mathcal O}_{{\mathbb C}{\mathbb P}^1}(2d+1)$. Indeed, if the nontrivial element of ${\mathbb Z}/2{\mathbb Z}$
acts as multiplication by $1$ on ${\mathcal O}_{{\mathbb C}{\mathbb P}^1}(2d+1)$, then the action of $\mathcal H\, \subset\,
\text{SL}(2,{\mathbb C})$ on ${\mathcal O}_{{\mathbb C}{\mathbb P}^1}(2d+1)$ produces an action of $H$
on ${\mathcal O}_{{\mathbb C}{\mathbb P}^1}(2d+1)$. On the other hand, since 
the exact sequence in \eqref{e10} does not split, Corollary \ref{cor2} says that ${\mathcal O}_{{\mathbb C}{\mathbb P}^1}(2d+1)$
does not admit any $H$--equivariance structure. Therefore, we conclude that the nontrivial element of
${\mathbb Z}/2{\mathbb Z}\, \subset\, \mathcal H$ acts on ${\mathcal O}_{{\mathbb C}{\mathbb P}^1}(2d+1)$ as multiplication by $-1$.

Now, take any finite dimensional complex $\mathcal H$--module $M$ satisfying the condition that the nontrivial element
of ${\mathbb Z}/2{\mathbb Z}\, \subset\, \mathcal H$ acts on $M$ as multiplication by $-1$. Consider the holomorphic vector
bundle ${\mathcal O}_{{\mathbb C}{\mathbb P}^1}(2d+1)\bigotimes_{\mathbb C} M$. The actions of $\mathcal H$ on
${\mathcal O}_{{\mathbb C}{\mathbb P}^1}(2d+1)$ and $M$ together produce an action of $\mathcal H$ on
${\mathcal O}_{{\mathbb C}{\mathbb P}^1}(2d+1)\bigotimes_{\mathbb C} M$. Note that the nontrivial
element of ${\mathbb Z}/2{\mathbb Z}\, \subset\, \mathcal H$ acts on both ${\mathcal O}_{{\mathbb C}{\mathbb P}^1}(2d+1)$
and $M$ as multiplication by $-1$. Thus the action of the subgroup ${\mathbb Z}/2{\mathbb Z}\, \subset\, \mathcal H$ on
${\mathcal O}_{{\mathbb C}{\mathbb P}^1}(2d+1)\bigotimes_{\mathbb C} M$ is the trivial one. Consequently, the
action of $\mathcal H$ on ${\mathcal O}_{{\mathbb C}{\mathbb P}^1}(2d+1)\bigotimes_{\mathbb C} M$ produces an action,
on ${\mathcal O}_{{\mathbb C}{\mathbb P}^1}(2d+1)\bigotimes_{\mathbb C} M$, of the quotient group $H$ in \eqref{e10}.

\begin{proposition}\label{prop4}\mbox{}
\begin{enumerate}
\item Assume that the exact sequence in \eqref{e10} splits. For any $d\, \in\, \mathbb Z$, let
$$
W \ =\ {\mathcal O}_{{\mathbb C}{\mathbb P}^1}(d)^{\oplus r}
$$
be an $H$--equivariant vector bundle on ${\mathbb C}{\mathbb P}^1$. Then
there is a finite dimensional complex $H$--module $M$ such that the $H$--equivariant vector bundle $W$
is holomorphically and $G$--equivariantly isomorphic to ${\mathcal O}_{{\mathbb C}{\mathbb P}^1}(d)
\bigotimes_{\mathbb C} M$.

\item Assume that the exact sequence in \eqref{e10} does not split. For any $d\, \in\, \mathbb Z$, let
$$
W \ =\ {\mathcal O}_{{\mathbb C}{\mathbb P}^1}(2d)^{\oplus r}
$$
be an $H$--equivariant vector bundle on ${\mathbb C}{\mathbb P}^1$. Then
there is a finite dimensional complex $H$--module $M$ such that the $H$--equivariant vector bundle $W$
is holomorphically and $G$--equivariantly isomorphic to ${\mathcal O}_{{\mathbb C}{\mathbb P}^1}(2d)
\bigotimes_{\mathbb C} M$.

\item Assume that the exact sequence in \eqref{e10} does not split. For any $d\, \in\, \mathbb Z$, let
$$
W \ =\ {\mathcal O}_{{\mathbb C}{\mathbb P}^1}(2d+1)^{\oplus r}
$$
be an $H$--equivariant vector bundle on ${\mathbb C}{\mathbb P}^1$. Then
there is a finite dimensional complex $\mathcal H$--module $M$, which satisfies the condition that the nontrivial element
of ${\mathbb Z}/2{\mathbb Z}\, \subset\, \mathcal H$ acts on $M$ as multiplication by $-1$, such that the $H$--equivariant
vector bundle $W$ is holomorphically and $G$--equivariantly isomorphic to ${\mathcal O}_{{\mathbb C}{\mathbb P}^1}(2d+1)
\bigotimes_{\mathbb C} M$.
\end{enumerate}
\end{proposition}

\begin{proof}
The proof of the first two statements is exactly identical to the proof of Proposition \ref{prop2}.

The proof of the third statement is also quite similar. The action of $H$ on $W$ produces an action of
$\mathcal H$ on $W$ using the projection $\beta'$ in \eqref{e10}. The actions of $\mathcal H$ on ${\mathcal O}_{{\mathbb C}
{\mathbb P}^1}(2d+1)$ and $W$ together produce an action of $\mathcal H$ on
$$
M\ :=\ H^0({\mathbb C}{\mathbb P}^1, \, \text{Hom}({\mathcal O}_{{\mathbb C}{\mathbb P}^1}(2d+1),\, W)).
$$
The nontrivial element of ${\mathbb Z}/2{\mathbb Z}\, \subset\, \mathcal H$ acts on $M$ as multiplication by $-1$ 
because it acts as multiplication by $-1$ (respectively, $1$) on ${\mathcal O}_{{\mathbb C}{\mathbb P}^1}(2d+1)$
(respectively, $W$).

The natural evaluation homomorphism
$$
\Psi\ :\ {\mathcal O}_{{\mathbb C}{\mathbb P}^1}(2d+1)\otimes_{\mathbb C} M\ \longrightarrow\, W
$$
that sends any $w\otimes m$, where $w\, \in\, {\mathcal O}_{{\mathbb C}{\mathbb P}^1}(2d+1)_x$,
$x\, \in\, {\mathbb C}{\mathbb P}^1$, and $m\, \in\, M$, to $m(w)\, \in\, W_x$, is an isomorphism, and it is
$\mathcal H$ equivariant. The subgroup ${\mathbb Z}/2{\mathbb Z}\, \subset\, \mathcal H$ acts trivially on
${\mathcal O}_{{\mathbb C}{\mathbb P}^1}(2d+1)\otimes_{\mathbb C} M$, because the nontrivial element of
${\mathbb Z}/2{\mathbb Z}\, \subset\, \mathcal H$ acts as multiplication by $-1$ on both
${\mathcal O}_{{\mathbb C}{\mathbb P}^1}(2d+1)$ and $M$.
\end{proof}

\begin{theorem}\label{thm2}\mbox{}
\begin{enumerate}
\item Assume that the exact sequence in \eqref{e10} splits. Any $H$--equivariant vector bundle $W$ on ${\mathbb C}{\mathbb P}^1$
is holomorphically and $H$--equivariantly isomorphic to an $H$--equivariant vector bundle of the form
$$
\bigoplus_{i=1}^r {\mathcal O}_{{\mathbb C}{\mathbb P}^1}(d_i)\otimes_{\mathbb C} M_i,
$$
where each $M_i$ is a finite dimensional complex $H$--module.
Moreover, the collection $\{d_i,\, M_i\}_{i=1}^r$ is uniquely determined by the isomorphism class of the
$H$--equivariant vector bundle $W$.

\item Assume that the exact sequence in \eqref{e10} does not splits. Any $H$--equivariant vector bundle $W$ on ${\mathbb C}{\mathbb P}^1$
is holomorphically and $H$--equivariantly isomorphic to an $H$--equivariant vector bundle of the form
$$
\left(\bigoplus_{i=1}^r {\mathcal O}_{{\mathbb C}{\mathbb P}^1}(2d_i)\otimes_{\mathbb C} M_i\right)
\bigoplus \left(\bigoplus_{i=1}^s {\mathcal O}_{{\mathbb C}{\mathbb P}^1}(2\delta_i+1)\otimes_{\mathbb C} N_i\right),
$$
where each $M_i$ (respectively, $N_i$) is a finite dimensional complex $H$--module (respectively, $\mathcal H$--module such that the action on $N_i$ of the nontrivial element of the subgroup $\mathbb Z/2\mathbb Z\,\subset\, \mathcal H$ is given by multiplication by -1).
Moreover, the collection $\{d_i,\, M_i\}_{i=1}^r\bigcup \{\delta_i,\, N_i\}_{i=1}^s$ is uniquely determined by the
isomorphism class of the $H$--equivariant vector bundle $W$.
\end{enumerate}
\end{theorem}

\begin{proof}
The proof is similar to the proof of Theorem \ref{thm1}.
\end{proof}


\end{document}